\documentclass[a4paper,11pt]{amsart}
\usepackage[utf8x]{inputenc}
\usepackage{graphicx, xcolor}
\usepackage{amsmath,amssymb,amsthm,comment,mathtools}
\usepackage[colorlinks,allcolors=blue]{hyperref} 
\usepackage{cleveref}
\usepackage{pdfpages}
\usepackage{amsmath}
\usepackage{subfigure,cite}
\usepackage{tikz}

\usepackage[a4paper,margin=2.6cm]{geometry}

\newcommand{\set}[1]{\left\{#1\right\}}
\newcommand{\sm}{\setminus}
\newcommand{\paren}[1]{\left(#1\right)}

\newcommand{\abs}[1]{\left\lvert #1 \right\rvert}
\newcommand{\FF}{\mathbb{F}}

\theoremstyle{plain}
\newtheorem{thm}{Theorem}[section]
\newtheorem{cor}[thm]{Corollary}
\newtheorem{prop}[thm]{Proposition}
\newtheorem{lem}[thm]{Lemma}

\theoremstyle{definition}
\newtheorem{ex}[thm]{Example}
\newtheorem{defi}[thm]{Definition}
\newtheorem{rmk}[thm]{Remark}

\numberwithin{equation}{section}
\title{Edge codes constructed from unicyclic graphs}

\author[S. Asensio]{Sara Asensio}
\address[S. Asensio] {Instituto de Investigaci\'on en Matem\'aticas (IMUVa), Universidad de Valladolid, Valladolid, Spain}
\email{sara.asensio@uva.es}

\author[G. Gaggero]{Giulia Gaggero}
\address[G. Gaggero] {Department of Mathematics and Statistics
McMaster University, Hamilton, ON L8S 4L8, Canada}
\email{gaggerog@mcmaster.ca}

\author[N. Ragunathan]{Naveena Ragunathan}
\address[N. Ragunathan]
{Department of Mathematics and Statistics
McMaster University, Hamilton, ON L8S 4L8, Canada}
\email{ragunatn@mcmaster.ca}

\author[A. Saha]{Abhilash Saha}
\address[A. Saha]
{Department of Mathematics and Statistics
McMaster University, Hamilton, ON L8S 4L8, Canada}
\email{sahaa25@mcmaster.ca}

\author[A. Van Tuyl]{Adam Van Tuyl}
\address[A. Van Tuyl]{Department of Mathematics and Statistics
McMaster University, Hamilton, ON L8S 4L8, Canada}
\email{vantuyla@mcmaster.ca}

\keywords{edge codes, evaluation codes, toric codes, unicyclic
graphs}
\subjclass[2020]{13P25, 5C65, 14G50, 94B27}
\date{\today}
\begin{document}

\begin{abstract} Jaramillo-Velez recently
introduced edge codes, a new class of toric evaluation codes
constructed from the edges of
a (hyper)graph $\mathcal{H}$.  In the
case that $\mathcal{H}$ is a tree,
Jaramillo-Velez computed both the 
minimum distance and the weight distribution of
the associated code.  
In this paper, we study edge codes associated to unicyclic graphs. 
Our most striking result is that computing the parameters of these codes is
subtle in the case that the induced cycle has an even length because
these values will depend on certain conditions regarding the 
length of the cycle and the size of the base field. 
\end{abstract}

\maketitle

%%%%%%%%%%%%%%%%%%%%%%%%%%%%%%%%%%%%%%

\section{Introduction}
Originally introduced by H{\o}holdt, van Lint and Pellikaan as a generalization of one-point algebraic geometry codes~\cite{HVLP98}, evaluation codes have evolved into a broad and active area of research. Evaluation
codes have been investigated in several algebraic frameworks, including affine variety codes and order domain codes~\cite{Gei08,GHM20}, and they are now studied in their own right~\cite{JVPV21}.

The codewords of an evaluation code are constituted by the evaluations of the elements of a polynomial vector space $\mathcal{L}\subseteq K[t_1,\ldots,t_s]$  on a geometric object $X \subseteq K^s$, where $K=\mathbb{F}_q$ is a finite field.  
A large family of evaluation codes is the family of toric codes introduced by Hansen in 2000~\cite{Han00}. In this case, $X$ is the affine torus $T=(K^*)^s$, and there are two main cases that have been studied with this choice of $X$. In the first case, the vector space $\mathcal{L}$ is derived from $d$-hypersimplices ~\cite{JVPV21,CP26,CP25,PS23}, and in the second case, it is generated by squarefree monomials~\cite{JVPV21,PS23}.

Recently, Jaramillo-Velez introduced a new family of toric codes~\cite{Jar25}. In this case $\mathcal{L}$ is parametrized by the set of edges of a hypergraph. 
Given any hypergraph $\mathcal{H}$ on the vertex set $V=\{t_1,\ldots,t_s\}$, its edges $E$ are subsets of $V$ of any cardinality. Given an edge $\{t_{i_1},\ldots,t_{i_\ell}\}\subseteq V$, one can associate to it a monomial $t_{i_1}\cdots t_{i_{\ell}}\in K[t_1,\ldots,t_s]$. Let $\mathcal{L}$ be the polynomial vector space $$KE(\mathcal{H}) = {\rm span}_{K}\{t_{i_1}\cdots t_{i_\ell} ~:~ \{t_{i_1},\ldots,t_{i_\ell}\} \in E\}.$$ Therefore, the \textit{evaluation map} for this family of codes is
$${\rm ev}: KE(\mathcal{H}) \rightarrow K^{|T|}  
,~~~f \mapsto (f(P_1),\ldots,f(P_{|T|})),$$ where $T=\{P_1,\dots,P_{|T|}\}$, and the corresponding \textit{edge toric code} is $\mathcal{C}_{\mathcal{H}} = {\rm ev}(KE(\mathcal{H}))\subseteq K^{|T|}$.

An edge toric code $\mathcal{C}_{\mathcal{H}}$ is a linear code of length $|T|$, i.e., it is a subspace of $K^{|T|}$. Thanks to the correspondence between the codewords and the polynomials of the $K$-vector space $KE(\mathcal{H})$, the dimension and the minimum distance of $\mathcal{C}_{\mathcal{H}}$ can also be related to properties of $\mathcal{H}$. In particular, it can be proved that studying the weight distribution of the edge code associated to a hypergraph $\mathcal H$ is equivalent to studying the zeroes of all polynomials in $KE(\mathcal H)$ over the affine torus. Studying zeroes of polynomials is a general problem in algebraic geometry and there are recent works focused on this, see e.g. ~\cite{AP25, NSJ26}.

Jaramillo-Velez provided estimations for the minimum distance of edge codes coming from $d$-uniform clutters and computed the minimum distances of all edge codes of graphs with five vertices. Additionally, 
given a tree $G$ on $s$ vertices, he computed the number of zeroes of every polynomial in $KE(G)$.   This allows one to compute the weight distribution
of $\mathcal{C}_G$ when $G$ is a tree. While this is not explicitly done
in \cite{Jar25}, we provide the details in Theorem \ref{thm:WeightDistTrees}.

The goal of this note is to extend Jaramillo-Velez's work
to a larger family, namely the family of {\it unicyclic graphs}, that is
graphs with exactly one induced cycle.  In particular,
we consider the problem of computing the minimum
distance and weight distribution of codes constructed from unicyclic graphs,
see Theorems \ref{thm:zeroesOddUnicyclic} and \ref{thm:zeroesEvenUnicyclic}.
Interestingly, the weight distribution of $\mathcal{C}_G$  will depend 
upon the parity of the size of the induced cycle in $G$.      
When the induced cycle has an odd
length, then the number of zeroes of $f \in KE(G)$ will only depend
upon the number of terms in $f$, but not the coefficients of $f$.
This result contrasts to the case that there is an even induced cycle in $G$
because the number of zeroes of $f \in KE(G)$ will depend both on
the number of terms in $f$ and its coefficients.

A similar question has been answered by Carvalho and Patanker in~\cite{CP26} for toric codes defined over hypersimplices. They computed the number of minimal and next-to-minimal codewords for a large family of hypersimplices. Notice that, even if their result can be considered more general because it applies to a larger number of codes, ours is more complete because we give the full weight distribution and not only the first two values of the distribution for
edge toric codes constructed from unicylic graphs.

The paper is organized as follows.    In Section \ref{sec.background} we 
recall the remaining background.  In Section \ref{sec:leaf},
we consider the relationship between the edge toric code
of a graph $G$ with a leaf $t$ with that of $G \setminus \{t\}$, the graph with 
the leaf removed.  This analysis leads to Theorem ~\ref{thm:attachingtrees}, 
a general technique to remove leaves from a graph $G$ to create 
a new graph $G''$ such that we can relate the zeroes of $f \in KE(G)$ with
those of a polynomial $g \in KE(G'')$.  In Section \ref{sec:cycles} we compute
the weight distribution of the edge toric code $\mathcal{C}_{C_s}$, where 
$C_s$ is an $s$-cycle. Finally, in Section \ref{sec:unicylic}, we combine
these results to describe the weight distribution of 
$\mathcal{C}_G$ for any unicylic graph $G$.

%%%%%%%%%%%%%%%%%%%%%%%%%%%%%%%%%%%%%

\section{Background}\label{sec.background}
We recall the background on edge toric codes
that we need for the later sections. While edge toric codes, as  introduced by Jaramillo-Velez \cite{Jar25}, 
can be defined  using any hypergraph, we only present 
the definition in the case of graphs. 
Throughout this paper, $K =\mathbb{F}_q$ denotes a finite
field  and $K^* = (\mathbb{F}_q)^*$ denotes
the multiplicative group of units of $K$. We also
make the assumption that $q > 2$ throughout this paper to avoid some
degenerate cases; for instance, see Example \ref{ex:charex}.

Let $G = (V,E)$ denote a finite simple graph
on the vertex set $V = \{t_1,\ldots,t_s\}$ and edge
set $E = \{e_1,\ldots,e_r\}$.  Each edge $e_i$ has the form $e_i = \{t_{i_1},t_{i_2}\}\subseteq V$ for
some $t_{i_1} \neq t_{i_2}$. 
We say that $t_{i_1}$ and $t_{i_2}$ are \textit{adjacent}. A graph $G$ is said to be a \textit{tree} if it is connected and contains no cycles, while a \textit{forest} is a disjoint union of trees. The \textit{cycle} of length $s$ is the graph $C_s$ whose vertex and egde sets are $V(C_s)=\{t_1,\dots,t_s\}$ and $E(C_s)=\{\{t_i,t_{i+1}\}\,:\,1\leq i\leq s-1\}\cup\{\{t_1,t_s\}\}$, respectively. A cycle is said to be \textit{induced} if there are no edges joining two non-consecutive vertices of the cycle. The main goal of this paper is to study \textit{unicyclic} graphs, which are graphs containing exactly one induced cycle.

Let $S = K[t_1,\ldots,t_s]$ 
denote
the polynomial ring over the field $K$.  Note that
we are abusing notation by letting $t_i$ represent
both a vertex of the graph and a variable in the polynomial ring. Using $G$, we define
the following $K$-vector space of polynomials:
$$KE(G) = {\rm span}_K\{t_it_j ~:~ \{t_i,t_j\} \in E\}.$$ 
Now consider the affine torus
$T= (K^*)^s$.   This torus has $|T| = (q-1)^s$ 
distinct points, and we put an order
on the elements of $T$, say 
$T = \{P_1,\ldots,P_{|T|}\}$.
We define the {\it evaluation map} to be
$${\rm ev}: KE(G) \rightarrow K^{|T|}  
~~\mbox{given by}~~f \mapsto (f(P_1),\ldots,f(P_{|T|})).$$
In other words, we evaluate $f$ at all points in $T$.

We can now recall the following definition due to Jaramillo-Velez
\cite[Definition A]{Jar25}.

\begin{defi}\label{def:EdgeCode}
The {\it edge (toric) code} $\mathcal C_G$ of the graph $G$ 
is the image of the evaluation map, that is, $\mathcal C_G := {\rm ev}(KE(G)) \subseteq K^{|T|}$. 
\end{defi}

A {\it linear code} of {\it length} 
$n$ and {\it dimension} $k$ is a subspace $\mathcal C \subseteq K^n$ 
where $\dim_K \mathcal C = k$.  Elements of $\mathcal C$ are called 
{\it codewords.}  Given a codeword
$c \in \mathcal C$, the {\it Hamming weight} of $c$ is the number
of non-zero coordinates in $c$.    For $1 \leq  i \leq n$,
let $A_i(\mathcal C)$ be the number of codewords in $\mathcal C$ with Hamming weight
$i$.  The vector $(A_1(\mathcal C),\ldots,A_n(\mathcal C))$ is the 
{\it weight distribution} of $\mathcal C$.  The {\it minimum 
distance} of $\mathcal C$ is the smallest $i$ such
that $A_i(\mathcal C) \neq 0$.  We denote the minimum 
distance by $\delta(\mathcal C)$. A linear code of length $n$, dimension $k$ and minimum distance $\delta$ is usually called an $[n,k,\delta]$-code.
Albeit all the previous definitions hold for any field $K$, the standard choice for coding theory is $K=\mathbb{F}_q$, where $\mathbb{F}_q$ is the finite field with $q$ elements.

As we mentioned in the introduction, for an edge code $\mathcal C_G$, Jaramillo-Velez 
related its length, dimension and minimum distance to 
properties of $G$.  To state his result,
we let $V_T(f)$ denote the affine
variety of all zeroes of $f$ in $T$, that is,
$$V_T(f) = \{P \in T ~:~ f(P) = 0\} \subseteq T\,.$$
With this notation, we have:

\begin{lem}\label{lem:EdgeCodesParameters}
Let $\mathcal C_G$ be the edge code of a graph $G = (V,E)$ on $s$ vertices, and let $T=(K^*)^s$. Then
\begin{enumerate}
    \item the length of $\mathcal C_G$ is $|T| = (q-1)^s$;
    \item the dimension of $\mathcal C_G$ is $|E|$; and
    \item the Hamming weight of the codeword corresponding to a polynomial $f\in KE(G)$ is $(q-1)^s-|V_{T}(f)|$, which implies that the minimum distance of $\mathcal C_G$ is 
    $$\delta(\mathcal C_G) = (q-1)^s-\max_{f\in KE(G)} |V_T(f)|.$$
    \item $A_{|T|}(\mathcal{C}_G)
    \geq |E|.$
\end{enumerate}
\end{lem}

\begin{proof}
    Statements $(1)-(3)$ can be found
    in \cite[Proposition 3.1]{Jar25}.
    For (4), notice that for any
    edge $\{t_i,t_j\} \in E$,
    we have $t_it_j \in KE(G)$ and
    $V_T(t_it_j) = \emptyset$, so the codeword associated to $t_it_j$ only has non-zero coordinates.
    Hence each $e \in E$ gives 
    a codeword of Hamming weight $|T|$.
\end{proof}

The formula for $\delta(\mathcal C_G)$
tells us that determining the minimum distance of the code $\mathcal{C}_G$ 
is equivalent to finding the largest number of roots in the affine torus a 
polynomial in $KE(G)$ can have. Furthermore, computing the number of roots of 
all polynomials in $KE(G)$ allows to determine the code's weight distribution. 
Our goal is to find these distributions
for all edge codes associated to unicyclic graphs.

When the graphs under consideration are trees, Jaramillo-Velez determined 
the number of roots of all polynomials associated to them.
We summarize these results since we shall also need them for the
unicyclic case.  In fact, we will give a new
proof for Proposition \ref{prop:DeliozeroesTrees} using
the techniques of the next section.

\begin{prop}[\hspace{-0.01cm}{\cite[Proposition 3.14]{Jar25}}] \label{prop:DeliozeroesTrees}
    Let $G$ be a tree with $s \geq 2$ vertices. For each $f\in KE(G)$ with $r\in\{1,2,\dots,s-1\}$ non-zero terms, \[|V_T(f)|=\sum_{i=1}^{r-1}(-1)^{i+1}(q-1)^{s-i}\,.\]

\end{prop}

From Lemma \ref{lem:EdgeCodesParameters} and  Proposition
\ref{prop:DeliozeroesTrees}, we can obtain the weight distribution of the edge code associated to a tree. 

\begin{thm}\label{thm:WeightDistTrees}
    Let $G$ be a tree with $s\geq 2$ vertices. The weight distribution of $\mathcal{C}_G$ satisfies
    \begin{eqnarray*}
        A_{\sum_{i=0}^{r-1}(-1)^{i}(q-1)^{s-i}}({\mathcal C}_{G})=\binom{s-1}{r}(q-1)^r \text{ for every }r\in\{1,2,\dots,s-1\},
    \end{eqnarray*} and $A_j(\mathcal C_{G})=0$ for the remaining values of $j$.

    Consequently, the minimum distance is given by $\delta(\mathcal{C}_G) = (q-1)^{s-1}(q-2)$, and it corresponds to $r=2$, that is, it is achieved with all polynomials
    with two non-zero terms.
\end{thm}

\begin{proof}
    Notice that a tree on $s$ vertices has $s-1$ edges, and there are ${s-1\choose r}(q-1)^r$ polynomials in $KE(G)$ with exactly $r$ non-zero terms.
\end{proof}

Jaramillo-Velez also proved the following result for graphs 
containing 4-cycles as  subgraphs.

\begin{thm}[\hspace{-0.01cm}{\cite[Theorem 3.7]{Jar25}}]
    Let $q>2$ and let $G$ be a finite simple graph on $s\geq 4$ vertices containing a 4-cycle 
    (not necessarily induced). Then the minimum distance of the edge code $\mathcal C_G$ is given by $$\delta(\mathcal C_G)=(q-2)^2(q-1)^{s-2}\,.$$
\end{thm}

In the following sections, we provide an extension of the previous 
result to more families of graphs. In particular, we work towards  studying
the family of unicyclic graphs.

%%%%%%%%%%%%%%%%%%%%%%%%%%%%%%%%%%%%

\section{Edge codes of graphs with a leaf}\label{sec:leaf}

We now examine the edge codes of graphs with a leaf.
Recall that we say that a vertex 
$t \in V$ of a graph $G$ is a 
{\it leaf} if there is exactly one edge $e \in E$ with 
$t \in e$. For any vertex $t \in V$, the graph 
$G' = G \setminus\{t\}$ is the graph formed by removing $t$ 
and all the  edges of $G$ that contain $t$.
The main result of this section is to relate
the edge code $\mathcal{C}_G$ with that
of $\mathcal{C}_{G'}$ when $t$ is a leaf.
As a consequence,  by repeatedly removing leaves, 
we can reduce the study of edge codes of 
graphs to smaller graphs.

We begin by relating the number of zeroes of a polynomial
$f \in KE(G)$ to those of a polynomial in $KE(G')$.
Some of the ideas of Lemma \ref{lem:AddingLeaves}
are implicit in the proof of 
\cite[Proposition 3.14]{Jar25} which only dealt
with trees.  However, the result applies more generally 
to any graph with a leaf.

\begin{lem}\label{lem:AddingLeaves}
    Let $G$ be a graph on vertices $\set{t_1, \ldots, t_s, t_{s+1}}$, 
    where $t_{s+1}$ is a leaf that is adjacent to $t_s$. 
    Set $G' = G \sm \set{t_{s+1}}$. Let $f \in KE(G)$, and 
    write $$f(t_1, \ldots, t_s, t_{s+1}) = 
    f'(t_1, \ldots, t_s) + \beta t_s t_{s+1},$$ with 
    $f' \in KE(G')$ and $\beta \in \FF_q.$ 
    Let $T = \paren{\FF_q^*}^{s+1}$ and $T' = \paren{\FF_q^*}^{s}$. Then
    $$\abs{V_T(f)} = 
    \begin{cases}
    (q-1) \abs{V_{T'}(f')} & \text{ if } \beta = 0, \ and \\
    (q-1)^s - \abs{V_{T'}(f')} & \text{ if } \beta \neq 0.
\end{cases}$$
\end{lem}
\begin{proof}
    If $\beta = 0$, then $f = f'$, which does not involve $t_{s+1}$. So, any choice of $a_{s+1} \in \FF_q^*$ and $(a_1, \ldots, a_s) \in V_{T'}(f')$ yields an element of $V_T(f)$. Thus, $$\abs{V_T(f)} = \abs{\FF_q^*} \cdot \abs{V_{T'}(f')} = (q-1) \abs{V_{T'}(f')}.$$
    
Now, assume $\beta \neq 0.$ Define the map 
$\phi: V_T(f) \to T' \sm V_{T'}(f')$ by 
$$\quad \phi(a_1, \ldots, a_s, a_{s+1}) = (a_1, \ldots, a_s).$$ 
We shall show that $\phi$ is a well-defined bijection.
To see that this map is well-defined,
suppose $(a_1, \ldots, a_{s+1}) \in  V_T(f)$. Then, 
$$f(a_1, \ldots, a_{s+1}) = 
f'(a_1, \ldots, a_s) + \beta a_s a_{s+1} = 0 \implies 
f'(a_1, \ldots, a_s) = - \beta a_s a_{s+1}.$$ 
Since $\beta, a_s, a_{s+1} \in \FF_q^*$, 
the right-hand side is non-zero.
Therefore, the left-hand side is also non-zero. In other words, 
$f'(a_1, \ldots, a_s) \neq 0$,
and thus $\phi(a_1,\ldots,a_{s+1}) \in T'\setminus V_{T'}(f')$.

To show the map is injective, 
suppose $\phi(a_1, \ldots, a_{s+1}) = 
\phi(a_1', \ldots, a_{s+1}')$. 
This directly gives $a_i = a_i'$ for $i=1,\ldots,s$.
Furthermore, $$a_{s+1} = -\dfrac{f'\paren{\phi(a_1, \ldots, a_{s+1})}}{\beta a_s} = -\dfrac{f'\paren{\phi(a_1', \ldots, a_{s+1}')}}{\beta a_s'} = a_{s+1}'.$$
Finally, for the surjectivity of $\phi$, given $(a_1, \ldots, a_s) \in T' \sm V_{T'}(f')$, let $$a_{s+1} = -\dfrac{f'\paren{a_1, \ldots, a_{s}}}{\beta a_s} \neq 0.$$ By construction, $\paren{a_1, \ldots, a_{s+1}} \in V_T(f)$, and $\phi\paren{a_1, \ldots, a_{s+1}} = (a_1, \ldots, a_s).$
\end{proof}

The above result can be used to derive some consequences
for the corresponding weight distributions.

\begin{thm}\label{thm:relateWeightDistributions}
     Let $G$ be a graph on vertices $\set{t_1, \ldots, t_s, t_{s+1}}$, 
    where $t_{s+1}$ is a leaf that is adjacent to $t_s$. 
    Set $G' = G \sm \set{t_{s+1}}$, and let $T = \paren{\FF_q^*}^{s+1}$ and $T' = \paren{\FF_q^*}^{s}$.

    Let $A_{\mathcal{C}_G} = (A_1(\mathcal{C}_G),\ldots,
    A_{|T|}(\mathcal{C}_G))$ be the weight distribution
    vector of $\mathcal{C}_G$, and similarly, let 
    $A_{\mathcal{C}_{G'}} = (A_1(\mathcal{C}_{G'}),\ldots,
    A_{|T'|}(\mathcal{C}_{G'}))$ be the weight distribution
    vector of $\mathcal{C}_{G'}$.
    Then:
    \begin{enumerate}
        \item If $A_i(\mathcal{C}_{G'}) \neq 0$, then 
        $A_i(\mathcal{C}_{G'}) \leq A_{i(q-1)}(\mathcal{C}_{G})$.
        \item If $A_i(\mathcal{C}_{G}) \neq 0$ and $i \neq |T|$, then
        either $A_{\frac{i}{(q-1)}}(\mathcal{C}_{G'}) \neq 0$ or
        $A_{|T|-i}(\mathcal{C}_{G'}) \neq 0$.
    \end{enumerate}
\end{thm}

\begin{proof}
    $(1)$ Suppose that $A_i(\mathcal{C}_{G'}) \neq 0$.  
    Let $B_i'$ be the set of all $f \in KE(G')$ with 
    $|T'|-|V_{T'}(f)| = i$. By definition, we have $|B_i'| = A_i(\mathcal{C}_{G'})$.
 Consider the natural injection $\phi:KE(G') \rightarrow KE(G)$
 given by $\phi(f) = f + 0t_st_{s+1} \in KE(G)$.  Let 
 $B_i = \phi(B'_i)$.   For any $\phi(f) \in B_i$, we have
 $|V_{T}(\phi(f))|= (q-1)|V_{T'}(f)|$ by Lemma \ref{lem:AddingLeaves}.  But then 
 $$|T|-|V_{T}(\phi(f))| = (q-1)|T'| - (q-1)|V_{T'}(f)|
 = (q-1)i.$$
 So, the codeword associated to every $\phi(f) \in B_i$ has $(q-1)i$ non-zero
 coordinates.   Hence
 $$A_{i(q-1)}(\mathcal{C}_{G}) \geq |B_i|
 = |B'_i| = A_i(\mathcal{C}_{G'}).$$

 $(2)$  Suppose that  $A_i(\mathcal{C}_{G}) \neq 0$ with $i \neq |T|$.  
 Thus there exists an $f \in KE(G)$ such that $|T|-|V_T(f)| = i$.  Because $i \neq |T|$,
 this implies that $f$ has more than one non-zero term. Indeed,
 if $f$ had a single non-zero term, then $V_T(f) = \emptyset$, and thus
 $i = |T|-|V_T(f)| = |T|.$
 Write $f$ as $f = f'(t_1,\ldots,t_s) + \beta t_st_{s+1}$.  
 If $\beta = 0$, then by Lemma \ref{lem:AddingLeaves} 
 $$i=|T| - |V_{T}(f)| = |T| -(q-1)|V_{T'}(f')| = 
 (q-1)\left(|T'| - |V_{T'}(f')|\right).$$
 Thus $\frac{i}{(q-1)} = |T'|-|V_{T'}(f')|$, and so 
 $A_{\frac{i}{(q-1)}}(\mathcal{C}_{G'}) \neq 0$.

 On the other hand, if $\beta \neq 0$, then by Lemma 
 \ref{lem:AddingLeaves} we have
 $$i=|T|-|V_{T}(f)| = |T|-(|T'|-|V_{T'}(f')|) \Longleftrightarrow
 |T|- i = |T'|-|V_{T'}(f')|.$$
 Note that $f'$ must have at least one non-zero 
 term, so $|V_{T'}(f')| \neq |T'|$. 
 But this then implies that $A_{|T|-i}(\mathcal{C}_{G'}) \neq 0$
 with $|T|-i \neq 0$.
\end{proof}

\begin{cor}\label{cor:mincodebounds}
     Let $G$ be a graph on vertices $\set{t_1, \ldots, t_s, t_{s+1}}$, 
    where $t_{s+1}$ is a leaf that is adjacent to $t_s$. 
    Set $G' = G \sm \set{t_{s+1}}$.   Then:
    \begin{enumerate}
        \item $\delta(\mathcal{C}_{G}) \leq (q-1)\delta(\mathcal{C}_{G'})$.
        \item $\delta(\mathcal{C}_{G'}) \leq \delta(\mathcal{C}_{G})/(q-1)$ or $\delta(\mathcal{C}_{G'}) \leq |T|-\delta(\mathcal{C}_{G})$.
    \end{enumerate}
    In particular, if there exists an $f \in KE(G)$ with 
    $f = f'(t_1,\ldots,t_s) + 0t_st_{s+1}$ such that 
    $\delta(\mathcal{C}_G) = |T| - |V_T(f)|$, then $\delta(\mathcal{C}_G)
    = (q-1)\delta(\mathcal{C}_{G'})$. 
\end{cor}

\begin{proof}
For the first part, apply Theorem \ref{thm:relateWeightDistributions} (1) to 
$i = \delta(\mathcal{C}_{G'})$.  The second part follows
from Theorem \ref{thm:relateWeightDistributions} (2) 
when $i = \delta(\mathcal{C}_{G})$.  As shown
within the proof of Theorem \ref{thm:relateWeightDistributions}, the
first of the two inequalities in the second part will occur if the coefficient of
$t_st_{s+1}$ in $f$ is zero.
\end{proof}

\begin{rmk}
As an easy observation, the first inequality in Corollary 
\ref{cor:mincodebounds} (2)
always holds if 
$|T|-\delta(\mathcal C_G)\leq\frac{\delta(\mathcal C_G)}{q-1}$, which
is equivalent to $\delta(\mathcal C_G)\geq |T|\frac{q-1}{q}$. When $q$ is arbitrarily large, this would mean $\delta(\mathcal C_G)$ is close to $|T|$, which is the length of the code. However, if we use the Singleton bound, this would imply low dimension (that is, a small number of edges in $G$) and this case wouldn't be very interesting. 
\end{rmk}

We now consider a more general situation.
Let $G$ be a graph on the vertex set
$\{t_1,\ldots,t_s\}$.  Suppose that
at the vertices $t_{i_1},\ldots,t_{i_n}$, 
a ``tree'' is 
attached.  More specifically, for each $j=1,\ldots,n$,
there are $s_j$ vertices $z_{j,1},\ldots,z_{j,s_j}$
such that the induced subgraph on 
$\{t_{i_j},z_{j,1},\ldots,z_{j,s_j}\}$ is 
a tree, which we denote as $T_{i_j}$. 
Note that each $T_{i_j}$ has $s_j$ edges and $s_j+1$ vertices.
We let $H$ denote the graph $G$
with the attached trees and write it as
 $H = G \cup T_{i_1} \cup \cdots \cup T_{i_n}$. In other words, 
 $H$ is the graph on the vertex set 
$$\{t_1,\ldots,t_s\} \cup 
\{z_{j,k} ~:~ 1 \leq j \leq n,~ 1 \leq k \leq s_j\}$$
with $s+s_1+\cdots+s_n$ vertices.
See   Figure \ref{fig:generalpicture} for a representation
of such a graph.

\begin{figure}[h!]
\centering
\includegraphics[width=8cm]{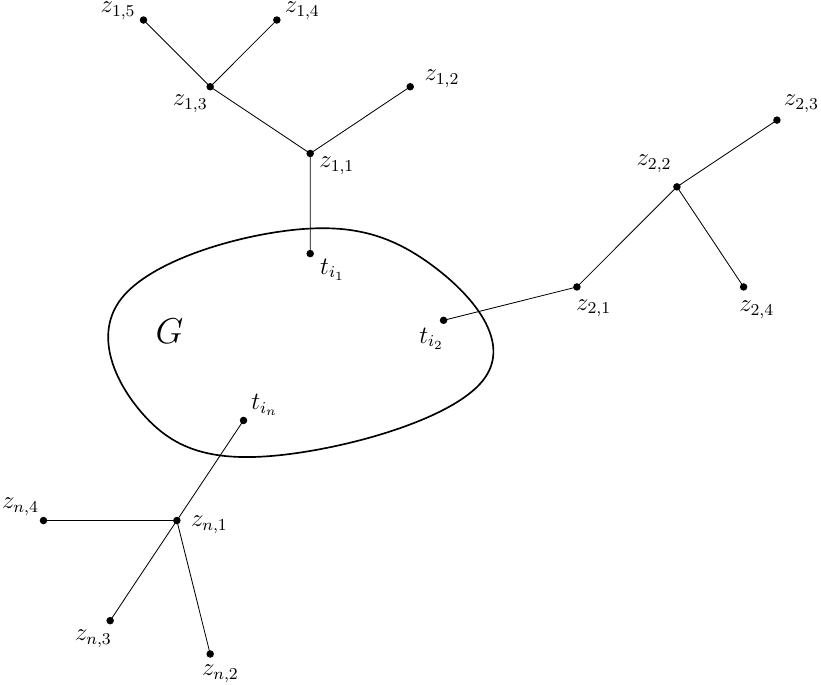}
\caption{A graph $G$ with  trees attached at
the vertices $t_{i_1},\ldots,t_{i_n}$}\label{fig:generalpicture}
\end{figure}

If $H$ is such a graph and $f \in KE(H)$, then
we can write $f$ as 
$$f(t_1,\ldots,t_s,z_{1,1},\ldots,z_{n,s_n}) =
g(t_1,\ldots,t_s) + g'(t_{i_1},\ldots,t_{i_n},
z_{1,1},\ldots,z_{n,s_n}),$$
with $g \in KE(G)$ and $g' \in 
KE(T_{i_1} \cup \cdots \cup T_{i_n})$.
The polynomial $g'$ can have up to $s_1 + \cdots + s_n$
non-zero terms since each tree $T_{i_j}$ has $s_j$ edges.
By repeated use of Lemma \ref{lem:AddingLeaves}, 
we can relate  the zeroes of $f$ to those of $g$.  
Consequently, the next result allows us to 
remove ``trees'' attached to $G$ if we are interested
in finding the number of zeroes of $f \in KE(H)$.

\begin{thm}\label{thm:attachingtrees}
Let $G$ be any graph on 
$s$ vertices.
    With the notation as above, let 
    $H = G \cup T_{i_1} \cup \cdots \cup T_{i_n}$,
    where
    $T_{i_j}$ is a tree with $s_j$ edges attached at vertex $t_{i_j}$ for 
    $j=1,\ldots,n$. For any $f \in KE(H)$,
    write
$$f(t_1,\ldots,t_s,z_{1,1},\ldots,z_{n,s_n}) =
g(t_1,\ldots,t_s) + g'(t_{i_1},\ldots,t_{i_n},z_{1,1},\ldots,z_{n,s_n}),$$
with $g \in KE(G)$ and $g' \in 
KE(T_{i_1} \cup \cdots \cup T_{i_n})$.
Suppose that $g'$ has $r$ non-zero terms.  
Then
    $$|V_{(\mathbb{F}_q^*)^m}(f)| = 
    (-1)^{r}(q-1)^{p-r}|V_{(\mathbb{F}_q^*)^s}(g)| + 
    \sum_{i=1}^{r}(-1)^{i+1}(q-1)^{m-i}\,,$$
    where $m=s+s_{1}+\cdots+s_{n}$ and 
    $p = s_{1}+\cdots + s_{n}$.   
\end{thm}

\begin{proof}
We begin by noting that the number of vertices of $H$ is $m$ and the 
number of edges in $T_{i_1} \cup \cdots \cup T_{i_n}$ is $p$. 
Thus, if $g'$ has $r$ non-zero terms, we must have $0 \leq r \leq p$.

We will do induction on $p$.  If $p=1$, then 
there is exactly one vertex
$t_{i_1}$ in $G$ with a tree attached to it, and a single vertex $z_{1,1}$ such that the induced graph on
$\{t_{i_1},z_{1,1}\}$ is a tree.   Thus, $z_{1,1}$ is a leaf and
we are in the situation of Lemma \ref{lem:AddingLeaves}.  
Hence, if $g'$
has $r=0$ non-zero terms, we have 
$$|V_{(\mathbb{F}_q^*)^{s+1}}(f)| = 
(q-1)|V_{(\mathbb{F}_q^*)^s}(g)| =
    (-1)^0(q-1)^{1-0}|V_{(\mathbb{F}_q^*)^s}(g)| + 
    \sum_{i=1}^{0}(-1)^{i+1}(q-1)^{m-i},$$
    where we use the convention that $\sum_{i=a}^b x_i = 0$ if $b<a$.
On the other hand, if $r=1$, the same lemma gives
$$|V_{(\mathbb{F}_q^*)^{s+1}}(f)| = (q-1)^s - |V_{(\mathbb F_q^*)^s}(g)|
= (-1)^{1}(q-1)^{1-1}|V_{(\mathbb F_q^*)^s}(g)|+ 
\sum_{i=1}^1(-1)^{i+1}(q-1)^{m-i}.$$
In both cases, the formula in the statement holds.
Note that in both of these formulas, we are using the fact that $m=s+1$.

We now assume that $p = s_{1}+\cdots +s_{n} > 1$.
Without loss of generality, we can assume that
all $s_{j} > 0$ (if $s_{j} = 0$, that would mean
that no tree is attached to $t_{i_j}$).   
Thus $H = G \cup T_{i_1} \cup \cdots \cup T_{i_n}$, where $T_{i_j}$ is a tree attached at $t_{i_j}$ and has $s_j$ edges.  In $T_{i_n}$, we can
find a vertex $z$ that is a leaf in $H$.  Suppose 
$z$ is adjacent to $z'$, so  
$zz' \in KE(T_{i_1}\cup \cdots \cup T_{i_n})$. 

Set $H' = H \setminus \{z\}$.
We now consider the cases: 1) the term $zz'$ appears
in $g'$ or 2) the term $zz'$ does not appear in $g'$.
In the first case, we have $f = g+g'$, and 
$g' = g''+\beta zz'$ with
$\beta \neq 0$. Here, 
$g''$ is a polynomial that does not involve $z$ and must
have $r-1$ non-zero terms. By 
Lemma \ref{lem:AddingLeaves} we have
$$|V_{(\mathbb{F}_q^*)^m}(f)| = (q-1)^{m-1} - 
|V_{(\mathbb{F}_q^*)^{m-1}}(g+g'')|.$$
Now $g+g'' \in KE(H')$, where the number of vertices
of $H'$ is $m-1$ and the number of edges among
the attached trees is $p-1$.  Since $g''$ has $r-1$ non-zero
terms, induction now implies that 
  $$|V_{(\mathbb{F}_q^*)^{m-1}}(g+g'')| = 
    (-1)^{r-1}(q-1)^{p-1-(r-1)}|V_{(\mathbb{F}_q^*)^s}(g)| + 
    \sum_{i=1}^{r-1}(-1)^{i+1}(q-1)^{m-1-i}.$$
Applying this substitution now gives
\begin{eqnarray*}
|V_{(\mathbb{F}_q^*)^m}(f)| &=& (q-1)^{m-1} - 
|V_{(\mathbb{F}_q^*)^{m-1}}(g+g'')| =\\
&=&
(q-1)^{m-1}- (-1)^{r-1}(q-1)^{p-1-(r-1)}|V_{(\mathbb{F}_q^*)^s}(g)| -
    \sum_{i=1}^{r-1}(-1)^{i+1}(q-1)^{m-1-i} =\\
    &=& (-1)^r(q-1)^{p-r}|V_{(\mathbb{F}_q^*)^s}(g)|
    + \sum_{i=1}^{r}(-1)^{i+1}(q-1)^{m-i}.
\end{eqnarray*}
And thus the formula holds in this case.

The proof for the second case is similar.  For completeness,
we work out all the details.  In this case, 
$g'= g'' + \beta zz' = g''$ since $\beta = 0$.  By
Lemma \ref{lem:AddingLeaves} we have
$$|V_{(\mathbb{F}_q^*)^m}(f)| = (q-1)|V_{(\mathbb{F}_q^*)^{m-1}}(g+g'')|.$$
As above, $g+g'' \in KE(H')$ where the number
of vertices of $H'$ is $m-1$ and the number
of edges among the attached trees is $p-1$.  Since
$g'$ has $r$ terms, so does $g''$.  Induction then
implies 
 $$|V_{(\mathbb{F}_q^*)^{m-1}}(g+g'')| = 
    (-1)^{r}(q-1)^{p-1-r}|V_{(\mathbb{F}_q^*)^s}(g)| + 
    \sum_{i=1}^{r}(-1)^{i+1}(q-1)^{m-1-i}\,.$$
Substituting this expression into our earlier
expression gives
\begin{eqnarray*}
    |V_{(\mathbb{F}_q^*)^m}(f)| &=& (q-1)|V_{(\mathbb{F}_q^*)^{m-1}}(g+g'')|= \\
    & = & (q-1)\left[(-1)^{r}(q-1)^{p-1-r}|V_{(\mathbb{F}_q^*)^s}(g)| + 
    \sum_{i=1}^{r}(-1)^{i+1}(q-1)^{m-1-i}\right]= \\
    &=& (-1)^r(q-1)^{p-r}|V_{(\mathbb{F}_q^*)^s}(g)|
    + \sum_{i=1}^{r}(-1)^{i+1}(q-1)^{m-i},
\end{eqnarray*}
as desired.
\end{proof}

One of the main results of \cite{Jar25} is now a corollary
of the previous result.

\begin{cor}[\hspace{-0.01cm}{\cite[Proposition 3.14]{Jar25}}]\label{cor:treeFormula}
If $G$ is a tree on $s \geq 2$ vertices, 
and if $f \in KE(G)$ 
has $r \in \{1,2,\ldots,s-1\}$ non-zero terms, then
$$|V_{(\mathbb{F}_q^*)^s}(f)| = \sum_{i=1}^{r-1}(-1)^{i+1}(q-1)^{s-i}.$$
\end{cor}

\begin{proof}
Let $t$ be any leaf of $G$ and let $G'$ be the graph
consisting of just the vertex $t$.  Then $G$
can be viewed as the graph obtained by
attaching a tree with $s-1$ edges to the graph
$G'$ at the vertex $t$.  
So $m=s$ and $p=s-1$.  We apply Theorem \ref{thm:attachingtrees},
where we highlight that we are applying the theorem to the graph
$G'$ with exactly one vertex.   If $f \in KE(G)$ has $r$ non-zero terms,
then by  Theorem \ref{thm:attachingtrees},
$$|V_{(\mathbb{F}_q^*)^s}(f)| = 
(-1)^r(q-1)^{s-1-r}|V_{(\mathbb{F}^*_q)^1}(g)|
+ \sum_{i=1}^r(-1)^{i+1}(q-1)^{s-i}.$$
We now  use the fact that $|V_{(\mathbb F_q^*)^1}(g)|=q-1$ for
    $g \in KE(G')=\{0\}$ to reduce 
    this expression to
    $$|V_{(\mathbb{F}_q^*)^s}(f)| = 
(-1)^r(q-1)^{s-1-r}(q-1)
+ \sum_{i=1}^r(-1)^{i+1}(q-1)^{s-i} = \sum_{i=1}^{r-1}(-1)^{i+1}(q-1)^{s-i},$$
as desired.
\end{proof}

%%%%%%%%%%%%%%%%%%%%%%%%%%%%%%%%%%%%
\section{Edge codes of cycles}\label{sec:cycles}

Before moving on to the study of arbitrary unicyclic graphs, we will begin by considering the simplest of them: cycles.

Given a positive integer $s$, we denote by $C_s$ the cycle of length $s$. To determine the weight distribution of the edge code $\mathcal C_{C_s}$, we consider the $K$-vector space $KE(C_s)$ and compute the number of zeroes in the affine torus of all polynomials in this $K$-vector space.

It is clear that if $f\in KE(C_s)$ has less than $s$ non-zero terms, then the edges corresponding to all non-zero terms in $f$ constitute a tree and the number of zeroes in the affine torus is already known in this case by Proposition \ref{prop:DeliozeroesTrees}
(or Corollary \ref{cor:treeFormula}).

Consequently, we have to focus on those polynomials whose terms correspond to all the edges of the cycle. The number of roots of these polynomials can be computed recursively, as we show in the next result:

\begin{prop}\label{prop:RecursiveFormulaCycles}
    For every $s\geq 3$, let $$f_s=\alpha_{1,2}^st_1t_2+\alpha_{2,3}^st_2t_3+\dots+\alpha_{s-1,s}^st_{s-1}t_s+\alpha_{1,s}^st_1t_s\,, \text{ with } \alpha_{i,j}^s\in\mathbb F_q^* \text{ for all } i,j\,, $$ be a polynomial associated to the cycle of length $s$. Then for all $s \geq 5$,
    $$|V_{(\mathbb F_q^*)^s}(f_s)|=q\ |V_{(\mathbb F_q^*)^{s-2}}(f_{s-2})|-(q-1)^{s-2}+\sum_{i=1}^{s-2}(-1)^{i+1}\ (q-1)^{s-i}\,,$$ where $\alpha_{i,i+1}^{s-2}=\alpha_{i+2,i+3}^{s}$ for all $i\in\{1,2,\dots,s-3\}$ and $\alpha_{1,s-2}^{s-2}=-\frac{\alpha_{2,3}^s\alpha_{1,s}^s}{\alpha_{1,2}^s}$.
\end{prop}
\begin{proof}
    To simplify the notation, let 
    \begin{eqnarray*}f_s& =& \alpha_{1,2}t_1t_2+\alpha_{2,3}t_2t_3+\dots+\alpha_{s-1,s}t_{s-1}t_s+\alpha_{1,s}t_1t_s \ =\\
    &=&t_1(\alpha_{1,2}t_2+\alpha_{1,s}t_s)+\alpha_{2,3}t_2t_3+\dots+\alpha_{s-1,s}t_{s-1}t_s\, ,
    \end{eqnarray*} 
    with $\alpha_{i,j}\in\mathbb F_q^*$ for all $i,j$.

    From the previous expression, it is easy to see that $f_s=0$ if and only if one of the two following situations holds:

    \smallskip

\noindent \underline{Case (1)}: $\alpha_{1,2}t_2+\alpha_{1,s}t_s=0$ and $\alpha_{2,3}t_2t_3+\dots+\alpha_{s-1,s}t_{s-1}t_s=0$.

In this situation, $t_2$ is determined by the value of $t_s$ and the second condition can be rewritten as $-\frac{\alpha_{2,3}\alpha_{1,s}}{\alpha_{1,2}}t_st_3+\dots+\alpha_{s-1,s}t_{s-1}t_s=0$, where the left-hand side is a polynomial associated to a cycle of length $s-2$. Since $t_1$ can be arbitrarily chosen, the number of zeroes of $f_s$ corresponding to this situation is $|V_{(\mathbb F_q^*)^s,1}(f_s)|=(q-1)\, |V_{(\mathbb F_q^*)^{s-2}}(f_{s-2})|$.
    
 \medskip

\noindent \underline{Case (2)}: $\alpha_{1,2}t_2+\alpha_{1,s}t_s\neq0$, $\alpha_{2,3}t_2t_3+\dots+\alpha_{s-1,s}t_{s-1}t_s\neq0$ and $t_1=-\frac{\alpha_{2,3}t_2t_3+\dots+\alpha_{s-1,s}t_{s-1}t_s}{\alpha_{1,2}t_2+\alpha_{1,s}t_s}$.

In this situation, $t_1$ will be determined by the remaining variables and we have to compute the number of $(s-1)$-tuples $(t_2,t_3,\dots,t_s)$ satisfying the first two conditions. If we denote $g:=\alpha_{1,2}t_2+\alpha_{1,s}t_s$ and $h:=\alpha_{2,3}t_2t_3+\dots+\alpha_{s-1,s}t_{s-1}t_s$, then the number of zeroes of $f_s$ corresponding to this situation is \begin{eqnarray*}
    |V_{(\mathbb F_q^*)^s,2}(f_s)|&=&(q-1)^{s-1}-|V_{(\mathbb F_q^*)^{s-1}}(g)\cup V_{(\mathbb F_q^*)^{s-1}}(h)|\ =\\
    &=& (q-1)^{s-1}-|V_{ (\mathbb F_q^*)^{s-1}}(g)|-|V_{(\mathbb F_q^*)^{s-1}}(h)|+|V_{(\mathbb F_q^*)^{s-1}}(g)\cap V_{(\mathbb F_q^*)^{s-1}}(h)|\ =\\
    &=&(q-1)^{s-1}-(q-1)^{s-2}-|V_{(\mathbb F_q^*)^{s-1}}(h)|+\frac{|V_{(\mathbb F_q^*)^s,1}(f_s)|}{q-1}\,.
\end{eqnarray*}

Since $h$ is a polynomial associated to a path of length $s-2$, 
which in particular is a tree, we can apply 
Proposition \ref{prop:DeliozeroesTrees} (or Corollary \ref{cor:treeFormula})
to get 
\begin{eqnarray*}
    |V_{(\mathbb F_q^*)^s,2}(f_s)|&=&
    (q-1)^{s-1}-(q-1)^{s-2}-\left(\sum_{i=1}^{s-3}(-1)^{i+1}(q-1)^{(s-1)-i}\right)+|V_{(\mathbb F_q^*)^{s-2}}(f_{s-2})|\ =\\
    &=&|V_{(\mathbb F_q^*)^{s-2}}(f_{s-2})|-(q-1)^{s-2}+\sum_{i=1}^{s-2}(-1)^{i+1}(q-1)^{s-i}\,.
\end{eqnarray*}

Combining the two situations, we get that 
\begin{eqnarray*}
|V_{(\mathbb F_q^*)^s}(f_s)|&=&
 |V_{(\mathbb F_q^*)^s,1}(f_s)|+|V_{(\mathbb F_q^*)^s,2}(f_s)|\ =\\
&=&q\ |V_{(\mathbb F_q^*)^{s-2}}(f_{s-2})|-(q-1)^{s-2}+
\sum_{i=1}^{s-2}(-1)^{i+1}\ (q-1)^{s-i}\,.
\end{eqnarray*} 
\end{proof}

This recursive formula for computing the number of zeroes in the affine torus 
$T$ of every polynomial associated to a cycle leads us to different results 
depending on whether $s$ is odd or even.  This fact will follow from the 
different behaviors of the codes constructed from $C_3$ and $C_4$.

\subsection{Odd cycles}
In the case of cycles of odd length, we first obtain an explicit formula 
for the number of zeroes in the affine torus of every polynomial 
associated to the complete cycle.

\begin{prop}\label{prop:zeroesOddCycles}
    Let $s\geq 3$ be an odd number, and let $f_s$ be a polynomial whose terms are in one-to-one correspondence with all the edges of $C_s$. Then $$|V_{(\mathbb F_q^*)^s}(f_s)|=\sum_{i=1}^{s-1}(-1)^{i+1}(q-1)^{s-i}\,.$$
\end{prop}

\begin{proof}
    We proceed by induction on $s$.

    Assume first that $s=3$. In this case, 
    $$f_3=\alpha_{1,2}t_1t_2+\alpha_{2,3}t_2t_3+\alpha_{1,3}t_1t_3=t_1 (\alpha_{1,2}t_2+\alpha_{1,3}t_3)+\alpha_{2,3}t_2t_3\,,$$ 
    with $\alpha_{i,j}\in\mathbb F_q^*$ for all $i,j$.  Then 
    $f=0$ if and only if $\alpha_{1,2}t_2+\alpha_{1,3}t_3\neq 0$ and 
    $t_1=-\frac{\alpha_{2,3}t_2t_3}{\alpha_{1,2}t_2+\alpha_{1,3}t_3}$. Hence $t_1$ is determined by the remaining variables, and there is a forbidden value for $t_2$. This implies that $$|V_{(\mathbb F_q^*)^3}(f_3)|=(q-1)(q-2)=(q-1)^2-(q-1)=\sum_{i=1}^2(-1)^{i+1}(q-1)^{3-i}\,.$$

    Assume now that the result is true for every odd number smaller than some odd number $s$, and let us prove it for $s$. Applying Proposition \ref{prop:RecursiveFormulaCycles} and the induction hypothesis for $|V_{(\mathbb F_q^*)^{s-2}}(f_{s-2})|$, we get \begin{eqnarray*}
        |V_{(\mathbb F_q^*)^s}(f_s)|&=&q\ |V_{(\mathbb F_q^*)^{s-2}}(f_{s-2})|-(q-1)^{s-2}+\sum_{i=1}^{s-2}(-1)^{i+1}\ (q-1)^{s-i}\ =\\
        &=&((q-1)+1)\sum_{i=1}^{s-3}(-1)^{i+1}(q-1)^{s-2-i}-(q-1)^{s-2}+\sum_{i=1}^{s-2}(-1)^{i+1}\ (q-1)^{s-i}\ =\\
        %&=&\sum_{i=1}^{n-3}(-1)^{i+1}(q-1)^{n-1-i}+\sum_{i=1}^{n-3}(-1)^{i+1}(q-1)^{n-2-i}-\\&&-(q-1)^{n-2}+\sum_{i=1}^{n-2}(-1)^{i+1}\ (q-1)^{n-i}=\\
        &=&\sum_{i=2}^{s-2}(-1)^{i}(q-1)^{s-i}+\sum_{i=1}^{s-3}(-1)^{i+1}(q-1)^{s-2-i}- (q-1)^{s-2}\\
        &&\ \ +\, (q-1)^{s-1}+\sum_{i=2}^{s-2}(-1)^{i+1}\ (q-1)^{s-i}\ =\\
        &=&\sum_{i=1}^{s-1}(-1)^{i+1}(q-1)^{s-i}\,. 
        \end{eqnarray*} 
        \end{proof}

If we now combine Lemma \ref{lem:EdgeCodesParameters} and Propositions \ref{prop:DeliozeroesTrees} and \ref{prop:zeroesOddCycles}, we can completely determine the weight distribution of the edge code associated to an odd cycle:

\begin{thm}\label{thm:WeightDistributionOddCycles}
    Let $s\geq 3$ be an odd number. If $f\in KE(C_s)$ has $r$ non-zero
    terms with $r\in\{1,2,\dots,s\}$, then $$|V_{(\mathbb F_q^*)^s}(f)|=\sum_{i=1}^{r-1}(-1)^{i+1}(q-1)^{s-i}\,.$$
    
    Then the weight distribution of $\mathcal C_{C_s}$ can be expressed as follows: \begin{eqnarray*}
        A_{\sum_{i=0}^{r-1}(-1)^i(q-1)^{s-i}}({\mathcal C}_{C_s})={s\choose r}(q-1)^r \text{ for every }r\in\{1,2,\dots,s\},
    \end{eqnarray*} and $A_j(\mathcal C_{C_s})=0$ for the remaining values of $j$.

    Consequently, the minimum distance of $\mathcal C_{C_s}$ corresponds to $r=2$, that is, 
    $$\delta(\mathcal{C}_{C_s}) = (q-1)^{s-1}(q-2),$$ and it 
    is achieved with all polynomials 
    $f \in KE(C_s)$ with two non-zero terms.
\end{thm}

\subsection{Even cycles} In the case of polynomials associated to cycles of odd length, we have shown in the previous subsection that all of them have the same number of zeroes in the corresponding affine torus, regardless of their coefficients. However, this is no longer true when considering cycles of even length. In this case, we need to distinguish two big families of polynomials depending on certain condition on the coefficients. We will show that all polynomials in each family have the same number of zeroes in the affine torus. The first example illustrating this behavior corresponds to the cycle of length four.

\begin{prop}\label{Prop:zeroesC4}
    Let $f=\alpha_{1,2}t_1t_2+\alpha_{2,3}t_2t_3+\alpha_{3,4}t_3t_4+\alpha_{1,4}t_1t_4$ be a polynomial associated to $C_4$, with $\alpha_{i,j}\in\mathbb F_q^*$ for all $i,j$. Then

    \[
    |V_{(\mathbb F_q^*)^4}(f)|=\begin{cases}
    (q-1)^3+(q-1)^2(q-2) & \text{if}\  \alpha_{1,2}\alpha_{3,4}-\alpha_{2,3}\alpha_{1,4}=0 , \ \text{and}\\
        (q-1)^2(q-3) & \text{otherwise}.
    \end{cases}
    \] 
\end{prop}

\begin{proof}
Since $f=t_1(\alpha_{1,2}t_2+\alpha_{1,4}t_4)+t_3(\alpha_{2,3}t_2+\alpha_{3,4}t_4)$, it is clear that $f=0$ if and only if one of the two following situations holds:  

\noindent \underline{Case (1)}: $\alpha_{1,2}t_2+\alpha_{1,4}t_4=0$ and $\alpha_{2,3}t_2+\alpha_{3,4}t_{4}=0$.

This system of linear equations has a non-trivial solution 
if and only if $\alpha_{1,2}\alpha_{3,4}-\alpha_{2,3}\alpha_{1,4}=0$.
The number of zeroes of $f$ corresponding to this situation is $|V_{(\mathbb F_q^*)^4,1}(f)|=(q-1)^3$ since the value of $t_2$ is determined by the value of $t_4$.

\noindent \underline{Case (2)}: $\alpha_{1,2}t_2+\alpha_{1,4}t_4\neq0$, $\alpha_{2,3}t_2+\alpha_{3,4}t_{4}\neq0$ and $t_1=-\frac{t_3(\alpha_{2,3}t_2+\alpha_{3,4}t_4)}{\alpha_{1,2}t_2+\alpha_{1,4}t_4}$.

In this situation, $t_1$ will be determined by the remaining variables and we just have to consider the two conditions $t_2\neq-\frac{\alpha_{1,4}}{\alpha_{1,2}}t_4$ and $t_2\neq -\frac{\alpha_{3,4}}{\alpha_{2,3}}t_4$.

If $\frac{\alpha_{1,4}}{\alpha_{1,2}}=\frac{\alpha_{3,4}}{\alpha_{2,3}}$ (which is equivalent to the condition $\alpha_{1,2}\alpha_{3,4}-\alpha_{2,3}\alpha_{1,4}=0$ appearing in \underline{Case (1)}), then we only have one restriction for $t_2$ and the number of zeroes of $f$ corresponding to this situation is $|V_{(\mathbb F_q^*)^4,2.1}(f)|=(q-1)^2(q-2)$.

If $\frac{\alpha_{1,4}}{\alpha_{1,2}}\neq\frac{\alpha_{3,4}}{\alpha_{2,3}}$, then we have two restrictions for $t_2$ and the number of zeroes of $f$ corresponding to this situation is $|V_{(\mathbb F_q^*)^4,2.2}(f)|=(q-1)^2(q-3)$.

\medskip

Combining the two situations, we get that \begin{eqnarray*}
|V_{(\mathbb F_q^*)^4}(f)|  = \begin{cases}
    |V_{(\mathbb F_q^*)^4,1}(f)|+|V_{(\mathbb F_q^*)^4,2.1}(f)| & \text{if}\  \alpha_{1,2}\alpha_{3,4}-\alpha_{2,3}\alpha_{1,4}=0 , \ \text{and}\\
        |V_{(\mathbb F_q^*)^4,2.2}(f)| & \text{otherwise},
\end{cases}
\end{eqnarray*}
\noindent which is the desired result.
\end{proof}

Using Propositions \ref{prop:RecursiveFormulaCycles} and \ref{Prop:zeroesC4}, we can obtain formulas for the number of zeroes of all polynomials corresponding to even cycles:

\begin{prop}\label{prop:zeroesEvenCycles}
    Let $s=2k$ be an even number, and let $$f_s=\alpha_{1,2}^st_1t_2+\alpha_{2,3}^st_2t_3+\dots+\alpha_{s-1,s}^st_{s-1}t_s+\alpha_{1,s}^st_1t_s, \text{ with } \alpha_{i,j}^s\in\mathbb F_q^* \text{ for all } i,j,$$ be a polynomial associated to the cycle of length $s$. 
    
    If we denote $d_m=\left(\sum_{i=1}^{m-2}(-1)^{i+1}(q-1)^{m-i}\right)-(q-1)^{m-2}$, then

    \[
    |V_{(\mathbb F_q^*)^s}(f_s)|=\begin{cases}
        q^{k-2}[(q-1)^3+(q-1)^2(q-2)]+\sum_{i=0}^{k-3}q^i d_{s-2i} \\ \text{ if } \prod_{\substack{i\in\{1,\dots,2k-1\}\\ i\ odd}} \alpha_{i,i+1}^s + (-1)^{k-1}\alpha_{1,s}^s\prod_{\substack{i\in\{1,\dots,2k-2\}\\ i\ even}}\alpha_{i,i+1}^s=0, \ and \vspace{5mm}\\
         
        q^{k-2}(q-1)^2(q-3)+\sum_{i=0}^{k-3}q^i d_{s-2i}\ \ otherwise.
    \end{cases}
    \]    
\end{prop}
\begin{proof}
    We proceed by induction on $k$.  If $k=2$, this result is just 
    Proposition \ref{Prop:zeroesC4}.

    Assume now that the result is true for certain $k-1\geq 2$, and let us prove it for $k$. By Proposition \ref{prop:RecursiveFormulaCycles} we have that $|V_{(\mathbb F_q^*)^{2k}}(f_{2k})|=q\ |V_{(\mathbb F_q^*)^{2k-2}}(f_{2k-2})|+d_{2k}$, where $f_{2k-2}$ satisfies that $\alpha_{i,i+1}^{2k-2}=\alpha_{i+2,i+3}^{2k}$ for all $i\in\{1,2,\dots,2k-3\}$ and $\alpha_{1,2k-2}^{2k-2}=-\frac{\alpha_{2,3}^{2k}\alpha_{1,2k}^{2k}}{\alpha_{1,2}^{2k}}$. By the induction hypothesis, we can distinguish two different possibilities for $|V_{(\mathbb F_q^*)^{2k-2}}(f_{2k-2})|$:

    \noindent \underline{Case (1)}: If $\prod_{\substack{i\in\{1,\dots,2k-3\}\\ i\ odd}} \alpha_{i,i+1}^{2k-2} + (-1)^{k-2}\alpha_{1,2k-2}^{2k-2}\prod_{\substack{i\in\{1,\dots,2k-4\}\\ i \ even}}\alpha_{i,i+1}^{2k-2}=0$, which is equivalent to $\prod_{\substack{i\in\{1,\dots,2k-1\}\\ i\ odd}} \alpha_{i,i+1}^{2k} + (-1)^{k-1}\alpha_{1,2k}^{2k}\prod_{\substack{i\in\{1,\dots,2k-2\}\\ i\ even}}\alpha_{i,i+1}^{2k}=0$ by taking into account the description of the coefficients $\alpha_{i,j}^{2k-2}$, then \[|V_{(\mathbb F_q^*)^{2k-2}}(f_{2k-2})|=
        q^{k-3}[(q-1)^3+(q-1)^2(q-2)]+\sum_{i=0}^{k-4}q^i d_{2k-2-2i}\,.\]

Therefore, in this case
\begin{eqnarray*}
    |V_{(\mathbb F_q^*)^{2k}}(f_{2k})| &=& q\ |V_{(\mathbb F_q^*)^{2k-2}}(f_{2k-2})|+d_{2k}\ =\\
    &=& q^{k-2}[(q-1)^3+(q-1)^2(q-2)]+\sum_{i=0}^{k-4}q^{i+1}d_{2k-2(i+1)} + d_{2k}\ = \\
    &=& q^{k-2}[(q-1)^3+(q-1)^2(q-2)]+\sum_{i=0}^{k-3}q^{i}d_{2k-2i}\,.
\end{eqnarray*}

\noindent 
\underline{Case (2)}: If the condition in \underline{Case (1)} does not hold, then 
\[|V_{(\mathbb F_q^*)^{2k-2}}(f_{2k-2})|=
        q^{k-3}(q-1)^2(q-3)+\sum_{i=0}^{k-4}q^i d_{2k-2-2i}\] 
and similarly as we did above, we get 
$|V_{(\mathbb F_q^*)^{2k}}(f_{2k})| = q^{k-2}(q-1)^2(q-3)+\sum_{i=0}^{k-3}q^{i}d_{2k-2i}\,.$
\end{proof}

It is clear from the previous proposition that the value of $|V_{(\mathbb F_q^*)^s}(f_s)|$ is larger when the condition $\prod_{\substack{i\in\{1,\dots,2k-1\}\\ i\ odd}} \alpha_{i,i+1}^s + (-1)^{k-1}\alpha_{1,s}^s\prod_{\substack{i\in\{1,\dots,2k-2\}\\ i\ even}}\alpha_{i,i+1}^s=0$ is satisfied. Now, to determine the minimum distance of $\mathcal{C}_{C_s}$ in this case we first need to determine the largest among the two quantities  $q^{k-2}[(q-1)^3+(q-1)^2(q-2)]+\sum_{i=0}^{k-3}q^i d_{s-2i}$ and $(q-1)^{s-1}$, where the latter corresponds to the number of zeroes of all polynomials with two non-zero terms. In order to compare these two quantities, we first present the following technical lemma which allows us to obtain a more explicit representation of the first quantity.

\begin{lem}\label{lem:TechnicalLemma}
    If we denote $d_m=\left(\sum_{i=1}^{m-2}(-1)^{i+1}(q-1)^{m-i}\right)-(q-1)^{m-2}$, then $$\sum_{i=0}^{k-3}q^i d_{2k-2i}=\frac{(q-1)^{2k}-(q-1)(q^{k-2}-1)}{q}-q^{k-3}(q-1)^4\,.$$
\end{lem}

\begin{proof}
Let $r=q-1$. Since
\begin{eqnarray*}
    \sum_{i=1}^{m-2}(-1)^{i+1}r^{m-i}=r^{m-1}\sum_{i=0}^{m-3}\left(-\frac{1}{r}\right)^i=\frac{r^m+(-1)^{m-1}r^2}{q}\,,
\end{eqnarray*} 
then $$d_{2k-2i}=\frac{r^{2k-2i}-r^2}{q}-r^{2k-2i-2}=\frac{r^{2k-2i-2}(r^2-q)-r^2}{q}\,.$$ 

To finish, \begin{eqnarray*}
    \sum_{i=0}^{k-3}q^i d_{2k-2i}&=&\sum_{i=0}^{k-3}q^i\left(\frac{r^{2k-2i-2}(r^2-q)-r^2}{q} \right)\ =\\
    &=&\frac{r^2-q}{q}\,r^{2k-2}\,\sum_{i=0}^{k-3}\left(\frac{q}{r^2}\right)^i-\frac{r^2}{q}\sum_{i=0}^{k-3}q^i\ =\\
    & = & \frac{r^2-q}{q}\,r^{2k-2}\,\frac{r^2}{q-r^2}\cdot\frac{q^{k-2}-r^{2k-4}}{r^{2k-4}}
    -\frac{r^2}{q}\cdot\frac{q^{k-2}-1}{q-1}\ =
    \\
    &=&\frac{r^{2k}}{q}-r^4q^{k-3}-\frac{r}{q}(q^{k-2}-1)\,.
 \end{eqnarray*}\end{proof}

If we now combine Lemmas \ref{lem:EdgeCodesParameters} and \ref{lem:TechnicalLemma} and Propositions \ref{prop:DeliozeroesTrees} and \ref{prop:zeroesEvenCycles}, we can completely determine the weight distribution of the edge code associated to an even cycle:

\begin{thm}\label{thm:WeightDistributionEvenCycles}
    Let $s=2k\geq 4$ be an even number. Then the weight distribution of $\mathcal C_{C_s}$ can be expressed as follows:
    \begin{eqnarray*}
        \text{ for every }r\in\{1,2,\dots,s-1\},\ A_{\sum_{i=0}^{r-1}(-1)^i(q-1)^{s-i}}({\mathcal C}_{C_s})&=&{s\choose r}(q-1)^r,\\
        A_{(q-1)^s-q^{k-2}[(q-1)^3+(q-1)^2(q-2)]-\frac{(q-1)^s-(q-1)(q^{k-2}-1)-q^{k-2}(q-1)^4}{q}}(\mathcal C_{C_s})&=& (q-1)^{s-1},\\
        A_{(q-1)^s-q^{k-2}(q-1)^2(q-3)-\frac{(q-1)^s-(q-1)(q^{k-2}-1)-q^{k-2}(q-1)^4}{q}}(\mathcal C_{C_s})&=& (q-1)^s-(q-1)^{s-1},\text{ and}
    \end{eqnarray*} $A_j(\mathcal C_{C_s})=0$ for the remaining values of $j$.
.

    Consequently, the minimum distance of $\mathcal C_{C_s}$ is \begin{eqnarray*}
        \delta(\mathcal{C}_{C_s})=\begin{cases}
            (q-1)^s-q^{k-2}[(q-1)^3+(q-1)^2(q-2)]-\frac{(q-1)^{s}-(q-1)(q^{k-2}-1)}{q}+q^{k-3}(q-1)^4 \\
            \ \ \text{ if } q^k(q-2)+1\geq (q-1)^{2k-2},\text{ and}    \vspace{5mm}\\

            (q-1)^{s-1}(q-2) \text{ otherwise}.
        \end{cases}
    \end{eqnarray*}
\end{thm}
\begin{proof}
    It is enough to notice that the first case holds when $$q^{k-2}[(q-1)^3+(q-1)^2(q-2)]+\frac{(q-1)^{s}-(q-1)(q^{k-2}-1)}{q}-q^{k-3}(q-1)^4\geq (q-1)^{s-1}\, , $$ and this is equivalent to $q^k(q-2)+1\geq (q-1)^{2k-2}$ by using the identity \begin{eqnarray*}
        q^{k-2}[(q-1)^3+(q-1)^2(q-2)]-q^{k-3}(q-1)^4&=&q^{k-2}(q-1)^2(2q-3)-q^{k-3}(q-1)^4\ =\\ &=&q^{k-3}(q-1)^2(q^2-q-1).
    \end{eqnarray*} Indeed,
    \begin{eqnarray*}
    q^k(q-2)+1 \geq (q-1)^{2k-2} & \Leftrightarrow & q^k(q-2) +q^{k-2} - q^{k-2} + 1  \geq (q-1)^{2k-2}\\
    &\Leftrightarrow & q^{k-2}(q^2(q-2)+1)-(q^{k-2}-1) \geq (q-1)^{2k-2}\\
    &\Leftrightarrow & q^{k-2}(q-1)(q^2-q-1)-(q^{k-2}-1) \geq (q-1)^{2k-2} \\
    &\Leftrightarrow & q^{k-2}(q-1)^2(q^2-q-1)-(q-1)(q^{k-2}-1) \geq (q-1)^{2k-1}.
    \end{eqnarray*}
    Now using the identity $q(q-1)^{2k-1}-(q-1)^{2k}=(q-1)^{2k-1}$ for 
    the right-hand side and rearranging, we get
    $$q^{k-2}(q-1)^2(q^2-q-1)+(q-1)^{2k} - (q-1)(q^{k-2}-1) \geq q(q-1)^{2k-1}.$$
    Dividing by $q$ then gives
    $$ q^{k-3}(q-1)^2(q^2-q-1)+\frac{(q-1)^{s}-(q-1)(q^{k-2}-1)}{q}\geq (q-1)^{s-1}\, ,$$
    and the first identity now completes the proof.
\end{proof}

\begin{ex}\label{ex:charex}
Recall that we are assuming that $q>2$ throughout this paper.  We 
make this assumption since many of our formulas fail for
$q=2$, partly because the affine torus only has
one point in this case. For example, let $G=C_3$ and $T=(\mathbb F_2^*)^3=\{(1,1,1)\}$. Then we only have two codewords, $(0)$ and $(1)$, but we have several polynomials giving rise to each of them, such as $x_1x_2$, $x_2x_3$ and $x_1x_3$. In this case,
$A_1(\mathcal{C}_G) =1$ and this does not coincide with our formula.
If $q>2$, then different polynomials give rise to different codewords and we do not run into these issues.
\end{ex}

\begin{ex}
In Theorem \ref{thm:WeightDistributionEvenCycles}, the 
condition that $q^k(q-2)+1\geq (q-1)^{2k-2}$ can only
hold for small values of $k$ since the term on the
right-hand side is a polynomial in $q$ of degree $2k-2$,
but the left-hand side is a polynomial of degree $k+1$.  
In fact, if $k \geq 5$, then the inequality only holds
for the prime $q=2$, the case we are excluding.  So
if $s \geq 10$, then $\delta(\mathcal{C}_s) = (q-1)^{s-1}(q-2)$.  If $k =4$, then $q^k(q-2)+1\geq (q-1)^{2k-2}$ 
if and only if $q =2,3$. Because we are excluding the
case $q=2$, when $q=3$ we have 
$$\delta(\mathcal{C}_{C_8}) =
(3-1)^8-3^{4-2}[(3-1)^3+(3-1)^2(3-2)]-\frac{(3-1)^{8}-(3-1)(3^{4-2}-1)}{3}+3^{4-3}(3-1)^4 = 116,
$$ which is less than $(3-1)^7(3-2) = 128$.

Note that this means that when $q=3$, we can find a polynomial
$f \in KE(C_8)$ with $(3-1)^8 - 116 = 140$ zeros.
Indeed, if we take the polynomial
$$f=t_1t_2+t_2t_3+t_3t_4+t_4t_5+t_5t_6+t_6t_7+t_7t_8 + t_1t_8$$
with all coefficients equal to one, then since $k=4$ 
we are in the first condition of Proposition ~\ref{prop:zeroesEvenCycles} and 
$$|V_{(\mathbb{F}_3^*)^8}(f) | = 3^{4-2}[(3-1)^3+(3-1)^2(3-2)]
+ \sum_{i=0}^{4-3}3^id_{8-2i} = 140.$$
This was double checked using {\it Macaulay2}.
\end{ex}

%%%%%%%%%%%%%%%%%%%%%%%%%%%%%%%%%%%%

\section{Edge codes of unicyclic graphs}\label{sec:unicylic}

Using the results in the previous sections, we determine the weight 
distribution of all edge codes associated to unicyclic graphs, 
i.e., to graphs containing exactly one induced cycle. Since these graphs 
can be obtained by successively adding leaves to a cycle, our main tool is Theorem \ref{thm:attachingtrees}.

\begin{thm}\label{thm:zeroesOddUnicyclic}
    Let $G$ be a unicyclic graph on $s$ vertices
    with induced cycle $C_n$ with $n$ odd. Then for any $f\in KE(G)$ with $r$ non-zero terms, 
    $$|V_{(\mathbb F_q^*)^s}(f)|=\sum_{i=1}^{r-1}(-1)^{i+1}(q-1)^{s-i}\,.$$
    
     Then the weight distribution of $\mathcal C_{G}$ can be expressed as follows: \begin{eqnarray*}
        A_{\sum_{i=0}^{r-1}(-1)^i(q-1)^{s-i}}({\mathcal C}_{G})={s\choose r}(q-1)^r \text{ for every }r\in\{1,2,\dots,s\},
    \end{eqnarray*} and $A_j(\mathcal C_{G})=0$ for the remaining values of $j$.

    Consequently, the minimum distance of $\mathcal C_{G}$ corresponds to $r=2$, that is, 
    $$\delta(\mathcal{C}_{G}) = (q-1)^{s-1}(q-2),$$ and it 
    is achieved with all polynomials 
    $f \in KE(G)$ with two non-zero terms.
\end{thm}

\begin{proof}
    We denote by $n\leq s$ the length of the induced cycle of $G$ and proceed by induction on $s$.
    If $n=s$, then $G$ is an odd cycle and the result follows from Theorem \ref{thm:WeightDistributionOddCycles}.

    Assume now that the result is true for certain $s\geq n$ and let us prove it for $s+1$. In this case, $G$ can be obtained from a unicyclic graph on $s$ vertices with an odd induced cycle of length $n$ by adding one leaf to any of its vertices. Thus, every $f\in KE(G)$ with $r$ non-zero terms is of the form $f(t_1,\dots,t_{s},t_{s+1})=f'(t_1,\dots,t_s)+\beta t_it_{s+1}$ for some $i\in\{1,\dots,s\}$, $\beta\in\mathbb F_q$ and $f'\in KE(G')$, where $G'$ is a unicyclic graph on $s$ vertices with an odd induced cycle of length $n$.

    By Lemma \ref{lem:AddingLeaves} we have that $$|{V_{(\mathbb F_q^*)^{s+1}}(f)}| = \begin{cases}
    (q-1) |{V_{(\mathbb F_q^*)^s}(f')}| & \text{ if } \beta = 0, \text{ and} \\
    (q-1)^s - |{V_{(\mathbb F_q^*)^s}(f')}| & \text{ if } \beta \neq 0\,.\end{cases}$$
    
    Now we distinguish two possibilities and use the induction hypothesis to finish the first part
    of the proof.
        If $\beta=0$, then $f'$ has $r$ non-zero terms. Consequently,  $$|V_{(\mathbb F_q^*)^{s+1}}(f)|=(q-1)\sum_{i=1}^{r-1}(-1)^{i+1}(q-1)^{s-i}=\sum_{i=1}^{r-1}(-1)^{i+1}(q-1)^{(s+1)-i}\,.$$

    If $\beta\neq 0$, then $f'$ has $r-1$ non-zero terms and $$|V_{(\mathbb F_q^*)^{s+1}}(f)|=(q-1)^s-\sum_{i=1}^{r-2}(-1)^{i+1}(q-1)^{s-i}=\sum_{i=1}^{r-1}(-1)^{i+1}(q-1)^{(s+1)-i}\,.$$

    Because $G$ has $s$ vertices and is
    unicyclic, it has also $s$ edges. 
    The formula for the weight distribution
    now follows from the fact that
    there are $\binom{s}{r}(q-1)^r$ polynomials
    in $KE(G)$ with exactly $r$ non-zero terms.
    \end{proof}

\begin{thm}\label{thm:zeroesEvenUnicyclic}
 Let $G$ be a unicyclic graph on $s$ vertices
    with induced cycle $C_n$ with $n$ even.
    Suppose that after relabeling, the
    induced graph on $\{t_1,\ldots,t_n\}$ is $C_n$
    with edges $\{\{t_i,t_{i+1}\} : i =1,\ldots,n\}
    \cup \{\{t_1,t_n\}\}$.
    Let $f \in KE(G)$ be a polynomial with
    $f = f_n + f'$, where 
    $$f_n=\alpha_{1,2}^nt_1t_2+\alpha_{2,3}^nt_2t_3+\dots+\alpha_{n-1,n}^nt_{n-1}t_n+\alpha_{1,n}^nt_1t_n\text{ and } \alpha_{i,j}^n\in\mathbb F_q \text{ for all } i,j.$$
    Suppose that $f = f_n +f'$ has $r$ non-zero terms.  

    \begin{enumerate}
        \item If at least one of the 
        coefficients of $f_n$ is
        zero, then 
         $$|V_{(\mathbb F_q^*)^s}(f)|=\sum_{i=1}^{r-1}(-1)^{i+1}(q-1)^{s-i}\,.$$
        \item If all the coefficients of $f_n$ are 
        non-zero, then 
         $$|V_{(\mathbb{F}_q^*)^s}(f)| = 
    (-1)^{r-n}(q-1)^{s-r}|V_{(\mathbb{F}_q^*)^n}(f_n)| + 
    \sum_{i=1}^{r-n}(-1)^{i+1}(q-1)^{s-i}\,,$$
    with $|V_{(\mathbb{F}_q^*)^n}(f_n)|$ given
    as in Proposition \ref{prop:zeroesEvenCycles}. 
    \end{enumerate}
\end{thm}

\begin{proof}
    Let $f = f_n +f' \in KE(G)$ be a polynomial
    with $r$ non-zero terms.
    
    If at least one of the coefficients of $f_n$ 
    is zero, then the terms of $f$ correspond to 
    a tree with $r$ edges.  More precisely, if
    the term $t_it_{i+1}$ (or $t_1t_n$) 
    does not appear in $f_n$,
    then $f$ can be viewed as a polynomial in
    $KE(G')$ where $G'$ is the graph $G$ with the
    edge $\{t_i,t_{i+1}\}$ (or $\{t_1,t_n\}$) removed.  Then $G'$ is a
    tree, also on $s$ vertices, so the result
    then follows by Proposition \ref{prop:DeliozeroesTrees} (or Corollary \ref{cor:treeFormula}).

    Suppose that all the coefficients of
    $f_n$ are non-zero, which means that $f'$ has $r-n$ non-zero terms.  Note that we can view
    $G$ as an $n$-cycle with trees attached at various vertices. In the
    notation of Theorem \ref{thm:attachingtrees}, our
    cycle $C_n$ plays the role of $G$, and we have
    a collection of trees on $s-n$ vertices attached to
    $G$.  Again, translating the statement
    of Theorem \ref{thm:attachingtrees} to our
    situation, we are letting $f_n = g$, $f'=g'$,
    $m=s$ and $p=s-n$.   Making these 
    substitutions, we get the formula in the statement.

    The computation of 
    $|V_{(\mathbb{F}_q^*)^n}(f_n)|$ will depend
    upon the coefficients of $f_n$, as described
    in Proposition \ref{prop:zeroesEvenCycles}.
\end{proof}

\begin{rmk}
By using Theorem \ref{thm:zeroesEvenUnicyclic}
and Theorem \ref{thm:WeightDistributionEvenCycles} one
can explicitly write out the weight 
distribution $A(\mathcal{C}_G)$ also in the case of
unicyclic graphs where the induced cycle has
an even length.  The formulas are almost the 
same as those in Theorem \ref{thm:WeightDistributionEvenCycles},
but one needs to use the facts that
if $1 \leq r \leq n-1$, then there
are $\binom{s}{r}(q-1)^r$ polynomials
with $r$ non-zero terms which correspond to trees
in $G$; if $n \leq r \leq s$, then there are 
$\binom{s}{r}(q-1)^r - \binom{s-n}{r-n}(q-1)^r$ polynomials
with $r$ non-zero terms which correspond to trees
in $G$;
there are $\binom{s-n}{r-n}(q-1)^{r-1}$ polynomials
with $r$ non-zero terms that correspond to
an induced subgraph that contains the cycle
and such that $f_n$ falls into the first case of Proposition
\ref{prop:zeroesEvenCycles}; and 
$\binom{s-n}{r-n}[(q-1)^r-(q-1)^{r-1}]$ polynomials
with $r$ non-zero terms that correspond to an induced subgraph that contains
the cycle and such that
$f_n$ falls into
the other case of Proposition \ref{prop:zeroesEvenCycles}.
We leave the details for an interested reader.
\end{rmk}

%%%%%%%%%%%%%%%%%%%%%%%%%%%%%%%%%%%
\section*{Acknowledgements}
The results in this paper were inspired by 
experiments in {\it Macaulay2} \cite{M2}; some of our code was
created using Claude (Opus 4.7).
Asensio is partially supported by the grant PID2022-137283NB-C22 funded by MICIU/AEI/10.13039/501100011033 and by ERDF/EU; and by European Social Fund Plus, Programa Operativo de Castilla y León and Junta de Castilla y León via its Consejería de Educación (and also via the project with reference CLU-2025-1-02-IMUVA). 
Gaggero is funded by SNF Postdoc.Mobility grant P500PT\textunderscore225436.
Van Tuyl’s research is supported by NSERC Discovery Grant 2024-05299.

\bibliographystyle{plain}
\bibliography{biblio}

\end{document}